\documentclass[10pt]{article}

\usepackage{amsmath,amsthm,amsfonts,latexsym,amsopn,verbatim,amscd,amssymb}

\theoremstyle{plain}
\newtheorem{theorem}{Theorem}[section]

\newtheorem{proposition}[theorem]{Proposition}
\newtheorem{corollary}[theorem]{Corollary}
\theoremstyle{definition}
\numberwithin{equation}{section}

\DeclareMathOperator{\L-spec}{L-Spec}

\newcommand{\bnum}{\begin{enumerate}}
\newcommand{\enum}{\end{enumerate}}

\begin{document}

\title{Laplacian  Spectrum of non-commuting  graphs of  finite groups }
\author{Parama Dutta, Jutirekha Dutta  and Rajat Kanti Nath\footnote{Corresponding author}}
\date{}
\maketitle
\begin{center}\small{ Department of Mathematical Sciences,\\ Tezpur
University,  Napaam-784028, Sonitpur, Assam, India.\\
Emails: parama@gonitsora.com, jutirekhadutta@yahoo.com and rajatkantinath@yahoo.com }
\end{center}


\smallskip

\noindent {\small{\textbf{Abstract:}  In this paper, we  compute the Laplacian  spectrum of  non-commuting graphs of some classes of finite non-abelian groups. Our computations reveal that the non-commuting graphs of all the groups considered in this paper are L-integral. We also obtain some conditions on a group $G$ so that its non-commuting graph is L-integral. 
}}

\bigskip



\section{Introduction} \label{S:intro}
Let $G$ be a finite group with centre $Z(G)$. The non-commuting graph of a non-abelian group $G$, denoted by ${\mathcal{A}}_G$, is a simple undirected graph whose vertex set is $G\setminus Z(G)$ and two vertices $x$ and $y$ are adjacent if and only if $xy \ne yx$.   Various   aspects of non-commuting graphs of different finite groups can be found in  \cite{Ab06,AF14,dbb10,Abd13,tal08}.  In \cite{Abd13}, Elvierayani and Abdussakir have computed the Laplacian spectrum of the non-commuting graph of   dihedral groups $D_{2m}$ where $m$ is odd and suggested to consider the case when $m$ is even.   In  this paper, we  compute the Laplacian  spectrum of the non-commuting graph  of $D_{2m}$ for any $m \geq 3$ using a different method. Our method also enables  to compute the Laplacian  spectrum of the non-commuting graphs of  several  well-known families finite non-abelian groups such as the quasidihedral groups, generalized quaternion groups, some projective special linear groups, general linear groups etc. In a separate paper \cite{dN17}, we study  the Laplacian energy of  non-commuting graphs of the groups considered in this paper. 

For a graph ${\mathcal{G}}$ we write $\overline{\mathcal{G}}$ and $V({\mathcal{G}})$ to denote the complement of ${\mathcal{G}}$ and the  set of vertices of   ${\mathcal{G}}$ respectively.
Let $A({\mathcal{G}})$ and $D({\mathcal{G}})$ denote the adjacency matrix  and degree matrix of a graph ${\mathcal{G}}$ respectively. Then the Laplacian matrix  of ${\mathcal{G}}$ is given by  $L({\mathcal{G}})  = D({\mathcal{G}}) - A({\mathcal{G}})$. We write $\L-spec({\mathcal{G}})$  to denote the Laplacian  spectrum of ${\mathcal{G}}$ and  $\L-spec({\mathcal{G}}) = \{ \alpha_1^{a_1}, \alpha_2^{a_2}, \dots, \alpha_n^{a_n}\}$  where $\alpha_1 < \alpha_2 < \cdots < \alpha_n$ are the eigenvalues of   $L({\mathcal{G}})$ with multiplicities $a_1, a_2, \dots, a_n$ respectively. A graph ${\mathcal{G}}$ is called L-integral if $\L-spec({\mathcal{G}})$ contains only integers.   As a consequence of our results, it follows that the non-commuting graphs of all the groups considered in this paper are L-integral. It is worth mentioning that L-integral graphs are studied extensively in \cite{Abreu08, Kirkland07, Merries94}.


\section{Preliminary results}
 It is well-known that  $\L-spec(K_n) = \{0^1, n^{n - 1}\}$  where $K_n$ denotes the complete graph on $n$ vertices. Further, we have the following results.
\begin{theorem}
If $\mathcal{G} = l_1K_{m_1}\sqcup l_2K_{m_2}\sqcup\cdots  \sqcup l_kK_{m_k}$, where $l_iK_{m_i}$ denotes the disjoint union of $l_i$ copies of  $K_{m_i}$ for $1 \leq i \leq k$ and $m_1 < m_2 < \cdots < m_k$, then
\[
\L-spec({\mathcal{G}}) = \left\{0^{\sum_{i = 1}^k l_i}, m_1^{l_1(m_1 - 1)}, m_2^{l_2(m_2 - 1)}, \dots, m_k^{l_k(m_k - 1)}\right\}.
\]
\end{theorem}
\begin{theorem} {\rm \cite[Theorem 3.6]{M091}}
Let $\mathcal{G}$ be a graph  such that $\L-spec({\mathcal{G}}) = \{ \alpha_1^{a_1}, \alpha_2^{a_2}$, \dots, $\alpha_n^{a_n}\}$ then $\L-spec(\overline{\mathcal{G}})$ is given by
\[
\{0, (|V(\mathcal{G})| - \alpha_n)^{a_n}, (|V(\mathcal{G})| - \alpha_{n-1})^{a_{n-1}}, (|V(\mathcal{G})| - \alpha_{n-2})^{a_{n-2}}, \dots, (|V(\mathcal{G})|- \alpha_{1})^{a_{1} -1}\}.
\]
\end{theorem}
As a corollary of the above two theorems we have the following result.
\begin{corollary}\label{L-spec-3}
If $\mathcal{G} = l_1K_{m_1}\sqcup l_2K_{m_2}\sqcup\cdots  \sqcup l_kK_{m_k}$, where $l_iK_{m_i}$ denotes the disjoint union of $l_i$ copies of  $K_{m_i}$ for $1 \leq i \leq k$ and $m_1 < m_2 < \cdots < m_k$, then
\begin{align*}
\L-spec(\overline{\mathcal{G}}) =  & \{0, \left(\sum_{i=1}^k l_im_i - m_k\right)^{l_k(m_k - 1)}, \left(\sum_{i=1}^k l_im_i - m_{k - 1}\right)^{l_{k-1}(m_{k-1} - 1)},\\
&  \dots, \left(\sum_{i=1}^k l_im_i - m_{1}\right)^{l_{1}(m_{1} - 1)}, \left(\sum_{i=1}^k l_im_i\right)^{\sum_{i=1}^k l_i - 1}\}.
\end{align*}
\end{corollary}

   A group $G$ is called an AC-group if $C_G(x)$ is abelian for all $x \in G\setminus Z(G)$. Various aspects of AC-groups can be found in \cite{Ab06, das13, Roc75}. The following result gives the Laplacian spectrum of the non-commuting graph of a finite non-abelian AC-group.
\begin{theorem}\label{AC-group}
Let $G$ be a finite non-abelian   AC-group.  Then
\begin{align*}
\L-spec({\mathcal{A}}_G) =
 & \{0, (|G| - |X_n|)^{|X_n| - |Z(G)| - 1}, \dots,\\
 & (|G| - |X_1|)^{|X_1| - |Z(G)| - 1}, (|G| - |Z(G)|)^{n - 1}\}.
\end{align*}
where $X_1,\dots, X_n$ are the distinct centralizers of non-central elements of $G$ such that $|X_1| \leq \cdots \leq |X_n|$.
\end{theorem}

\begin{proof}
Let $G$ be a finite non-abelian   AC-group and   $X_i = C_G(x_i)$ where $x_i \in G \setminus Z(G)$ and $1\leq i \leq n$. Let $x, y \in X_i \setminus Z(G)$  for some $i$ and $x \ne y$  then, since $G$ an
AC-group, there is an edge between $x$ and $y$ in $\overline{{\mathcal{A}}_G}$.  Suppose that $x \in (X_i\cap X_j)\setminus Z(G)$ for some $1\leq i \ne j \leq n$. Then  $[x, x_i] = 1$ and $[x, x_j] = 1$. Let $s \in C_G(x)$ then $[s, x_i] = 1$ since $x_i \in C_G(x)$ and $G$ is an AC-group. Therefore, $s \in C_G(x_i)$ and so $C_G(x) \subseteq C_G(x_i)$.
Again, let  $t \in C_G(x_i)$ then $[t, x] = 1$ since $x  \in C_G(x_i)$ and $G$ is an AC-group. Therefore, $t \in C_G(x)$ and so $C_G(x_i) \subseteq C_G(x)$. Thus $C_G(x) = C_G(x_i)$.  Similarly, it can be seen that $C_G(x) = C_G(x_j)$, which is a contradiction.
Therefore, $X_i\cap X_j = Z(G)$ for any $1\leq i \ne j \leq n$. This shows that
\begin{equation}\label{ducomgraph}
\overline{{\mathcal{A}}_G} = \overset{n}{\underset{i = 1}{\sqcup}}K_{|X_i|-|Z(G)|}.
\end{equation}
Therefore, by Corollary \ref{L-spec-3}, we have
\begin{align*}
&\L-spec({\mathcal{A}}_G) =
  \{0, \left(\sum_{i = 1}^n(|X_i| - |Z(G)|) - (|X_n| - |Z(G)|)\right)^{|X_n| - |Z(G)| - 1}, \dots,\\
 & \left(\sum_{i = 1}^n(|X_i| - |Z(G)|) - (|X_1| - |Z(G)|)\right)^{|X_1| - |Z(G)| - 1}, \left(\sum_{i = 1}^n(|X_i| - |Z(G)|)\right)^{n - 1}\}.
\end{align*}
Hence, the result follows noting that $\underset{i = 1}{\overset{n}{\sum}}(|X_i| - |Z(G)|) = |G| - |Z(G)|$.
\end{proof}



\begin{corollary}\label{AC-cor}
Let $G$ be a finite non-abelian  AC-group and $A$ be any finite abelian group.  Then
\begin{align*}
\L-spec({\mathcal{A}}_{G\times A}) =&
  \{0, (|A|(|G| - |X_n|))^{|A|(|X_n| - |Z(G)|) - 1}, \dots,\\
 & (|A|(|G| - |X_1|))^{|A|(|X_1| - |Z(G)|) - 1}, (|A|(|G| - |Z(G)|))^{n - 1}\}.
\end{align*}
where $X_1,\dots, X_n$ are the distinct centralizers of non-central elements of $G$ such that $|X_1| \leq \cdots \leq |X_n|$.
\end{corollary}

\begin{proof}
It is easy to see that $G\times A$ is  an  AC-group and  $X_1\times A,  X_2 \times A,\dots, X_n \times A$  are the distinct centralizers of non-central elements of $G\times A$. Hence, the result follows from Theorem \ref{AC-group} noting that $Z(G\times A) = Z(G)\times A$.  \end{proof}

\section{Groups with given central factors}

In this section, we compute the  Laplacian spectrum of the  non-commuting graphs of some  families of finite non-abelian groups  whose central factors are some well-known finite  groups. We begin with the following.

\begin{theorem} \label{order-20}
Let $G$ be a finite group and $\frac{G}{Z(G)} \cong Sz(2)$, where $Sz(2)$ is the Suzuki group presented by $\langle a, b : a^5 = b^4 = 1, b^{-1}ab = a^2 \rangle$. Then
\[
\L-spec({\mathcal{A}}_G) =\{  0,  {(15|Z(G)|)}^{4|Z(G)|-1},  {(16|Z(G)|)}^{15|Z(G)|-5},  {(19|Z(G)|)}^5 \}.
\]
\end{theorem}
\begin{proof}
We have
\[
\frac{G}{Z(G)} = \langle aZ(G), bZ(G) : a^5Z(G) = b^4Z(G) = Z(G), b^{-1}abZ(G) = a^2Z(G) \rangle.
\]
 Observe that
\[
\begin{array}{ll}
C_G(ab) &= Z(G)\sqcup abZ(G) \sqcup a^4b^2Z(G)\sqcup a^3b^3Z(G),\\
C_G(a^2b) &= Z(G)\sqcup a^2bZ(G) \sqcup a^3b^2Z(G)\sqcup ab^3Z(G),\\
C_G(a^2b^3) &= Z(G)\sqcup a^2b^3Z(G) \sqcup ab^2Z(G)\sqcup a^4bZ(G),\\
C_G(b) &= Z(G)\sqcup bZ(G) \sqcup b^2Z(G)\sqcup b^3Z(G),\\
C_G(a^3b) &= Z(G)\sqcup a^3bZ(G) \sqcup a^2b^2Z(G)\sqcup a^4b^3Z(G) \quad \text{ and }\\
C_G(a)  &= Z(G)\sqcup aZ(G) \sqcup a^2Z(G)\sqcup a^3Z(G)\sqcup a^4Z(G)\\
\end{array}
\]
are the only centralizers of non-central elements of $G$. Also note that these centralizers are abelian subgroups of $G$. Thus $G$ is an  AC-group. 
We have $|C_G(a)| = 5|Z(G)|$ and
\[
|C_G(ab)| =  |C_G(a^2b)| =  |C_G(a^2b^3)| =  |C_G(b)| =  |C_G(a^3b)| = 4|Z(G)|.
\]
 Therefore, by   Theorem \ref{AC-group}, the result follows.
\end{proof}

\begin{theorem}\label{main2}
Let $G$ be a finite group such that $\frac{G}{Z(G)} \cong {\mathbb{Z}}_p \times {\mathbb{Z}}_p$, where $p$ is a prime integer. Then
\[
\L-spec({\mathcal{A}}_G) = \{ 0,  {((p^2-p)|Z(G)|)}^{(p^2-1)|Z(G)|-p-1},  {((p^2-1)|Z(G)|)}^p \}.
\]
\end{theorem}
\begin{proof}
Let $|Z(G)| = n$ then since $\frac{G}{Z(G)}\cong {\mathbb{Z}}_p\times {\mathbb{Z}}_p$ we have  $\frac{G}{Z(G)} = \langle aZ(G), bZ(G) : a^p, b^p, aba^{-1}b^{-1} \in Z(G)\rangle$, where $a, b \in G$ with $ab \ne ba$. Then for any $z \in Z(G)$, we have
\begin{align*}
C_G(a) &= C_G(a^iz) \,\,\, = Z(G) \sqcup aZ(G) \sqcup \cdots \sqcup a^{p -1}Z(G) \text{ for } 1 \leq i \leq p - 1,\\
C_G(a^jb) &= C_G(a^jbz) = Z(G) \sqcup a^jbZ(G) \sqcup \cdots \sqcup a^{(p -1)j}b^{p - 1}Z(G) \text{ for } 1 \leq j \leq p.
\end{align*}
These are the only  centralizers of non-central elements of $G$. Also note that these centralizers are abelian subgroups of $G$. Therefore, $G$ is an AC-group.
We have $|C_G(a)| =  |C_G(a^jb)| = pn$ for $1 \leq j \leq p$.
Hence, the result follows from Theorem \ref{AC-group}.
 \end{proof}
\noindent As a corollary we have the following result.
\begin{corollary}
Let $G$ be a non-abelian group of order $p^3$, for any prime $p$, then
\[
\L-spec({\mathcal{A}}_G) = \{  0 , {(p^3-p^2)}^{p^3-2p-1} , {(p^3-p)}^p\}.
\]
\end{corollary}

\begin{proof}
Note that $|Z(G)| = p$ and  $\frac{G}{Z(G)} \cong {\mathbb{Z}}_p \times {\mathbb{Z}}_p$. Hence the  result follows from Theorem \ref{main2}.
\end{proof}

\begin{theorem}\label{main4}
Let $G$ be a finite group such that $\frac{G}{Z(G)} \cong D_{2m}$, for $m \geq 2$. Then
\begin{align*}
\L-spec({\mathcal{A}_G}) = &\{ 0 , {(m|Z(G)|)}^{(m-1)|Z(G)|-1} , {(2(m-1)|Z(G)|)}^{m|Z(G)|-m},\\
&{((2m-1)|Z(G)|)}^m \}.
\end{align*}
\end{theorem}

\begin{proof}
Since $\frac{G}{Z(G)} \cong D_{2m}$ we have $\frac{G}{Z(G)} = \langle xZ(G), yZ(G) : x^2, y^m,  xyx^{-1}y\in Z(G)\rangle$, where $x, y \in G$ with $xy \ne yx$.
It is not difficult to see that  for any $z \in Z(G)$,
\[
C_G(xy^j) = C_G(xy^jz) = Z(G)  \sqcup xy^jZ(G), 1 \leq j \leq m
\]
and
\[
 C_G(y) = C_G(y^iz) = Z(G) \sqcup yZ(G) \sqcup\cdots \sqcup  y^{m - 1}Z(G), 1 \leq i \leq m - 1
\]
are the only  centralizers of non-central elements of $G$. Also note that these centralizers are abelian subgroups of $G$. Therefore, $G$ is an AC-group.
We have  $|C_G(x^jy)| = 2n$ for $1 \leq j \leq m$ and $|C_G(y)| = mn$, where $|Z(G)| = n$. Hence, the result follows from Theorem \ref{AC-group}.
\end{proof}

\noindent Using Theorem \ref{main4}, we now compute the Laplacian spectrum of the non-commuting graphs of the groups $M_{2mn}, D_{2m}$ and $Q_{4n}$ respectively.

\begin{corollary}\label{main05}
Let $M_{2mn} = \langle a, b : a^m = b^{2n} = 1, bab^{-1} = a^{-1} \rangle$ be a metacyclic group, where $m > 2$. Then $\L-spec({\mathcal{A}}_{M_{2mn}})$
\[
= \begin{cases}
\{ 0, (mn)^{mn-n-1}, {(2mn-2n)}^{mn-m}, {(2mn-n)}^m \} &  \text{if $m$ is odd}\\

\{ 0, (mn)^{mn-2n-1}, {(2mn-4n)}^{mn-\frac {m}{2}}, {(2mn-2n)}^\frac {m}{2}\} &  \text{if $m$ is even}.\\
\end{cases}
\]
\end{corollary}
\begin{proof}
Observe that $Z(M_{2mn}) = \langle b^2 \rangle$ or $\langle b^2 \rangle \cup a^{\frac{m}{2}}\langle b^2 \rangle$ according as $m$ is odd or even.  Also, it is easy to see that $\frac{M_{2mn}}{Z(M_{2mn})} \cong D_{2m}$ or $D_m$ according as $m$ is odd or even. Hence, the result follows from Theorem \ref{main4}.
\end{proof}
\noindent As a corollary to the above result we have the following result.
\begin{corollary}\label{main005}
Let $D_{2m} = \langle a, b : a^m = b^{2} = 1, bab^{-1} = a^{-1} \rangle$ be  the dihedral group of order $2m$, where $m > 2$. Then
\[
\L-spec({\mathcal{A}}_{D_{2m}}) = \begin{cases}
\{ 0, {m}^{m-2}, {(2m - 1)}^m \} & \text{if $m$ is odd}\\
\{ 0, {m}^{m-3}, {(2m - 4)}^\frac {m}{2}, (2m - 2)^\frac {m}{2} \} & \text{if $m$ is even}.\\
\end{cases}
\]
\end{corollary}

\begin{corollary}\label{q4m}
Let $Q_{4n} = \langle x, y : y^{2n} = 1, x^2 = y^n,xyx^{-1} = y^{-1}\rangle$, where $n \geq 2$, be the   generalized quaternion group of order $4n$. Then
\[
\L-spec({\mathcal{A}_{Q_{4n}}}) = \{ 0, {(2n)}^{2n - 3}, {(4n - 4)}^n, {(4n - 2)}^n \}.
\]
\end{corollary}
\begin{proof}
The result follows from Theorem \ref{main4} noting that  $Z(Q_{4n}) = \{1, a^n\}$ and  $\frac{Q_{4n}}{Z(Q_{4n})}$ $\cong D_{2n}$.
\end{proof}

\section{Some well-known groups}

 In this section,  we compute the Laplacian spectrum of the  non-commuting graphs of some well-known families of finite groups. We begin with the  family of finite groups having order $pq$ where $p$ and $q$ are primes.
\begin{proposition}\label{order-pq}
Let $G$ be a non-abelian group of order $pq$, where $p$ and $q$ are primes with $p\mid (q - 1)$. Then
\[
\L-spec({\mathcal{A}}_G) = \{0, {(pq - q)}^{q - 2}, {(pq - p)}^{pq - 2q}, {(pq - 1)^q}\}.
\]
\end{proposition}

\begin{proof}
It is easy to see that $|Z(G)| = 1$ and $G$ is an AC-group. Also the centralizers of non-central elements of $G$ are precisely the Sylow subgroups of $G$. The number of Sylow $q$-subgroups and Sylow $p$-subgroups of $G$ are one and $q$ respectively.
Hence,  the result follows from  Theorem \ref{AC-group}.
\end{proof}

\begin{proposition}\label{semid}
The Laplacian spectrum of the non-commuting graph of the quasidihedral group $QD_{2^n} = \langle a, b : a^{2^{n-1}} =  b^2 = 1, bab^{-1} = a^{2^{n - 2} - 1}\rangle$, where $n \geq 4$, is given by
\[
\L-spec({\mathcal{A}}_{QD_{2^n}}) = \{0, {(2^{n-1})}^{2^{n-1}-3}, {(2^n-4)}^{2^{n-2}}, {(2^n-2)}^{2^{n-2}}\}.
\]
\end{proposition}
\begin{proof}
It is well-known that $Z(QD_{2^n}) = \{1, a^{2^{n - 2}}\}$. Also
\[
C_{QD_{2^n}}(a) = C_{QD_{2^n}}(a^i) = \langle a \rangle \text{ for } 1 \leq i \leq 2^{n - 1} - 1, i \ne 2^{n - 2}
\]
and
\[
C_{QD_{2^n}}(a^jb) = \{1, a^{2^{n - 2}}, a^jb, a^{j + 2^{n - 2}}b \} \text{ for } 1 \leq j \leq 2^{n - 2}
\]
are the only  centralizers of non-central elements of $QD_{2^n}$. Note that these centralizers are abelian subgroups of  $QD_{2^n}$. Therefore, $QD_{2^n}$ is an AC-group.
We have $|C_{QD_{2^n}}(a)| = 2^{n - 1}$ and $|C_{QD_{2^n}}(a^jb)| = 4$ for $1 \leq j \leq 2^{n - 2}$.  Hence,  the result follows from  Theorem \ref{AC-group}.
\end{proof}

\begin{proposition}\label{psl}
The Laplacian spectrum of the non-commuting graph of the projective special linear group  $PSL(2, 2^k)$, where $k \geq 2$,   is given by
\begin{align*}
\L-spec({\mathcal{A}}_{PSL(2, 2^k)}) = & \{0, {(2^{3k}-2^{k+1}-1)}^{2^{3k-1}-2^{2k}+2^{k-1}}, {(2^{3k}-2^{k+1})}^{2^{2k}-2^k-2},\\
 & {(2^{3k}-2^{k+1}+1)}^{2^{3k-1}-2^{3k}-3.2^{k-1}}, {(2^{3k}-2^k-1)}^{2^{2k}+2^k}\}.
\end{align*}
\end{proposition}

\begin{proof}
We know that  $PSL(2, 2^k)$ is a non-abelian group of order $2^k(2^{2k} - 1)$ with trivial center. By Proposition 3.21 of \cite{Ab06}, the set of centralizers of non-trivial elements of $PSL(2, 2^k)$ is given by
\[
\{xPx^{-1}, xAx^{-1}, xBx^{-1} : x \in PSL(2, 2^k)\}
\]
where $P$ is an elementary abelian \quad $2$-subgroup and  $A, \quad B$ are  cyclic subgroups of  $PSL(2, 2^k)$ having order $2^k, 2^k - 1$ and $2^k + 1$ respectively. Also the number of conjugates of $P, A$ and $B$ in $PSL(2, 2^k)$ are  $2^k + 1, 2^{k - 1}(2^k + 1)$ and $2^{k - 1}(2^k - 1)$ respectively. Note that $PSL(2, 2^k)$ is a AC-group and so, by \eqref{ducomgraph}, we have
\[
\overline{{\mathcal{A}}_{PSL(2, 2^k)}} = (2^k + 1)K_{|xPx^{-1}| - 1} \sqcup 2^{k - 1}(2^k + 1)K_{|xAx^{-1}| - 1} \sqcup 2^{k - 1}(2^k - 1)K_{|xBx^{-1}| - 1}.
\]
That is, $\overline{{\mathcal{A}}_{PSL(2, 2^k)}} = (2^k + 1)K_{2^k - 1} \sqcup 2^{k - 1}(2^k + 1)K_{2^k - 2} \sqcup 2^{k - 1}(2^k - 1)K_{2^k}$. Hence, the result follows from   Corollary \ref{L-spec-3}.
\end{proof}

\begin{proposition}
The Laplacian spectrum of the non-commuting graph of the general linear group  $GL(2, q)$, where $q = p^n > 2$ and $p$ is a prime integer,   is given by
\begin{align*}
\L-spec({\mathcal{A}}_{GL(2, q)}) = & \{0, {(q^4-q^3-2q^2+q+1)}^{\frac {q^4-2q^3+q}{2}}, {(q^4-q^3-2q^2+2q)}^{q^3-q^2-2q}, \\
& {(q^4-q^3-2q^2+3q-1)}^{\frac {q^4-2q^3-2q^2+q}{2}}, {(q^4-q^3-q^2+1)}^{q^2+q}\}.
\end{align*}
\end{proposition}

\begin{proof}
We have $|GL(2, q)| = (q^2 -1)(q^2 - q)$ and $|Z(GL(2, q))| = q - 1$. By Proposition 3.26 of  \cite{Ab06}, the set of centralizers of non-central elements of $GL(2, q)$ is given by
\[
\{xDx^{-1}, xIx^{-1}, xPZ(GL(2, q))x^{-1} : x \in GL(2, q)\}
\]
where $D$ is the subgroup of $GL(2, q)$ consisting  of all diagonal matrices, $I$ is a  cyclic subgroup of $GL(2, q)$ having order $q^2 - 1$  and $P$ is the Sylow $p$-subgroup of $GL(2, q)$ consisting of all upper triangular matrices with $1$ in the diagonal. The orders of  $D$ and $PZ(GL(2, q))$ are  $(q - 1)^2$ and $q(q - 1)$ respectively. Also   the number of conjugates of $D, I$ and $PZ(GL(2, q))$ in $GL(2, q)$  are  $\frac{q(q + 1)}{2}, \frac{q(q - 1)}{2}$ and $q + 1$ respectively. Since $GL(2, q)$ is an AC-group (see Lemma 3.5 of \cite{Ab06}), by \eqref{ducomgraph}, we have $\overline{{\mathcal{A}}_{GL(2, q)}} =$
\[
 \frac{q(q + 1)}{2}K_{|xDx^{-1}| - q + 1} \sqcup \frac{q(q - 1)}{2}K_{|xIx^{-1}| - q + 1} \sqcup (q + 1)K_{|xPZ(GL(2, q))x^{-1}| - q + 1}.
\]
That is, $\overline{{\mathcal{A}}_{GL(2, q)}} = \frac{q(q + 1)}{2}K_{q^2 - 3q + 2} \sqcup \frac{q(q - 1)}{2}K_{q^2 - q} \sqcup (q + 1)K_{q^2 - 2q + 1}$. Hence, the result follows from   Corollary \ref{L-spec-3}.
\end{proof}

\begin{proposition}\label{Hanaki1}
Let $F = GF(2^n), n \geq 2$ and $\vartheta$ be the Frobenius  automorphism of $F$, that is, $\vartheta(x) = x^2$ for all $x \in F$. Then the Laplacian spectrum of the non-commuting graph of the group
\[
A(n, \vartheta) = \left\lbrace U(a, b) = \begin{bmatrix}
        1 & 0 & 0\\
        a & 1 & 0\\
        b & \vartheta(a) & 1
       \end{bmatrix} : a, b \in F \right\rbrace
\]
under matrix multiplication given by $U(a, b)U(a', b') = U(a + a', b + b' + a'\vartheta(a))$ is
\[
\L-spec({\mathcal{A}}_{A(n, \vartheta)}) = \{ 0, {(2^{2n}-2^{n+1})}^{{(2^n-1)}^2}, {(2^{2n}-2^n)}^{2^n-2} \}.
\]
\end{proposition}

\begin{proof}
Note that $Z(A(n, \vartheta)) = \{U(0, b) : b\in F\}$ and so $|Z(A(n, \vartheta))| = 2^n$. Let $U(a, b)$ be a non-central element of $A(n, \vartheta)$. It can be seen that the centralizer of $U(a, b)$ in $A(n, \vartheta)$ is $Z(A(n, \vartheta))\sqcup U(a, 0)Z(A(n, \vartheta))$. Clearly $A(n, \vartheta)$ is an AC-group and so, by \eqref{ducomgraph}, we have $\overline{{\mathcal{A}}_{A(n, \vartheta)}} = (2^n - 1)K_{2^n}$. Hence the result follows from     Corollary \ref{L-spec-3}.
\end{proof}

\begin{proposition}\label{Hanaki2}
Let $F = GF(p^n)$, $p$ be a prime. Then the Laplacian spectrum of the non-commuting graph of the group
\[
A(n, p) = \left\lbrace V(a, b, c) = \begin{bmatrix}
        1 & 0 & 0\\
        a & 1 & 0\\
        b & c & 1
       \end{bmatrix} : a, b, c \in F \right\rbrace
\]
under matrix multiplication $V(a, b, c)V(a', b', c') = V(a + a', b + b' + ca', c + c')$ is
\[
\L-spec({\mathcal{A}}_{A(n, p)}) = \{ 0, {(p^{3n}-p^{2n})}^{p^{3n}-2p^{n}-1}, {(p^{3n}-p^{n})}^{p^n} \}.
\]
\end{proposition}

\begin{proof}
We have $Z(A(n, p)) = \{V(0, b, 0) : b \in F\}$ and so $|Z(A(n, p))| = p^n$. The centralizers of non-central elements of $A(n, p)$ are given by
\begin{enumerate}
\item If $b, c \in F$ and $c \ne 0$ then the centralizer of $V(0, b, c)$ in $A(n, p)$ is\\ $\{V(0, b', c') : b', c' \in F\}$ having order $p^{2n}$.
\item If $a, b \in F$ and $a \ne 0$ then the centralizer of $V(a, b, 0)$ in $A(n, p)$ is\\ $\{V(a', b', 0) : a', b' \in F\}$ having order $p^{2n}$.
\item If $a, b, c \in F$ and $a \ne 0, c \ne 0$ then the centralizer of $V(a, b, c)$ in $A(n, p)$ is $\{V(a', b', ca'a^{-1}) : a', b' \in F\}$ having order $p^{2n}$.
\end{enumerate}
It can be seen that all the centralizers of non-central elements of $A(n, p)$ are abelian. Hence $A(n, p)$ is an AC-group and so, by \eqref{ducomgraph}, we have
\[
\overline{{\mathcal{A}}_{A(n, p)}} = K_{p^{2n} - p^n}\sqcup K_{p^{2n} - p^n}\sqcup (p^n - 1)K_{p^{2n} - p^n} = (p^n + 1)K_{p^{2n} - p^n}.
\]
Hence the result follows  from   Corollary \ref{L-spec-3}.
\end{proof}

We would like to mention here that the groups considered in Proposition \ref{Hanaki1}-\ref{Hanaki2} are constructed by Hanaki (see \cite{Han96}). These groups are also considered in \cite{ali00}, in order to compute their numbers of distinct centralizers.

\section{Some consequences}

Note that the non-commuting graphs of  all the groups considered in Section 3 and 4 are L-integral. In this section, we determine some conditions on $G$ so that its non-commuting graph becomes L-integral.

A finite group is called an $n$-centralizer group if it has $n$ numbers of distinct element centralizers. It clear that $1$-centralizer groups are precisely the abelian groups. There are no $2$, $3$-centralizer finite groups. The study of these groups was initiated by  Belcastro and  Sherman   \cite{bG94} in the year 1994. We have the following results regarding $n$-centralizer groups.

\begin{proposition}\label{4-cent}
If $G$ is a finite $4$-centralizer group then ${\mathcal{A}}_G$ is L-integral.
\end{proposition}
\begin{proof}
Let $G$ be a finite $4$-centralizer group. Then, by  \cite[Theorem 2]{bG94}, we have  $\frac{G}{Z(G)} \cong {\mathbb{Z}}_2 \times {\mathbb{Z}}_2$. Therefore, by Theorem \ref{main2}, we have
\[\L-spec({\mathcal{A}}_G) = \{ 0,  {(2|Z(G)|)}^{3|Z(G)|-3},  {(3|Z(G)|)}^2 \}.
\]
 Hence, ${\mathcal{A}}_G$ is L-integral.
\end{proof}

\noindent Further, we have the following result.

\begin{proposition}
If $G$ is   a finite $(p+2)$-centralizer $p$-group for any prime $p$, then ${\mathcal{A}}_G$ is L-integral.
\end{proposition}
\begin{proof}
Let $G$ be a finite $(p + 2)$-centralizer $p$-group. Then, by   \cite[Lemma 2.7]{ali00}, we have  $\frac{G}{Z(G)} \cong {\mathbb{Z}}_p \times {\mathbb{Z}}_p$. Therefore, by Theorem \ref{main2}, we have
\[
\L-spec({\mathcal{A}}_G) = \{ 0,  {((p^2-p)|Z(G)|)}^{(p^2-1)|Z(G)|-p-1},  {((p^2-1)|Z(G)|)}^p \}.
\]
Hence, ${\mathcal{A}}_G$ is L-integral.
\end{proof}

\begin{proposition}\label{5-cent}
If $G$ is a  finite $5$-centralizer  group then ${\mathcal{A}}_G$ is L-integral.
\end{proposition}
\begin{proof}
Let $G$ be a finite $5$-centralizer group. Then by  \cite[Theorem 4]{bG94} we have  $\frac{G}{Z(G)} \cong {\mathbb{Z}}_3 \times {\mathbb{Z}}_3$ or $D_6$. Now, if $\frac{G}{Z(G)} \cong {\mathbb{Z}}_3 \times {\mathbb{Z}}_3$ then  by Theorem \ref{main2} we have
$\L-spec({\mathcal{A}}_G) = \{ 0,  {(6|Z(G)|)}^{8|Z(G)|-4},  {(8|Z(G)|)}^3 \}$ and  hence ${\mathcal{A}}_G$ is L-integral. If $\frac{G}{Z(G)} \cong D_6$ then, by Theorem \ref{main4}, we have
\[
\L-spec({\mathcal{A}_G}) = \{0,  (3|Z(G)|)^{2|Z(G)|-1},  (4|Z(G)|)^{3|Z(G)| - 3}, (5|Z(G)|)^3 \}
\]
  and hence ${\mathcal{A}}_G$ is L-integral.  Therefore, the result follows.
\end{proof}

\noindent We also have the  following corollary.
\begin{corollary}
Let $G$ be a finite non-abelian group and $\{x_1, x_2, \dots, x_r\}$ be a set of pairwise non-commuting elements of $G$ having maximal size. Then ${\mathcal{A}}_G$ is L-integral if $r = 3, 4$.

\end{corollary}
\begin{proof}
By Lemma 2.4 in \cite{ajH07}, we have that $G$ is a $4$-centralizer or a $5$-centralizer group according as  $r = 3$ or $4$. Hence the result follows from Proposition \ref{4-cent} and Proposition \ref{5-cent}.
\end{proof}

The commuting probability of a finite group $G$ denoted by $\Pr(G)$ is  the probability that any two randomly chosen elements of $G$ commute. Clearly, $\Pr(G) = 1$ if and only if $G$ is abelian. The study of $\Pr(G)$ is originated from a paper of Erd$\ddot{\rm o}$s and Tur$\acute{\rm a}$n \cite{Et68}. Various results on $\Pr(G)$ can be found in \cite{Caste10,Dnp13,Nath08}. The following results show that ${\mathcal{A}}_G$ is L-integral if  $\Pr(G)$ has some particular values.

\begin{proposition}
If $\Pr(G) \in \{\frac{5}{14}, \frac{2}{5}, \frac{11}{27}, \frac{1}{2}, \frac{5}{8}\}$ then ${\mathcal{A}}_G$ is L-integral.
\end{proposition}
\begin{proof}
If $\Pr(G) \in \{\frac{5}{14}, \frac{2}{5}, \frac{11}{27}, \frac{1}{2}, \frac{5}{8}\}$ then as shown in \cite[pp. 246]{Rusin79} and \cite[pp. 451]{Nath13}, we have $\frac{G}{Z(G)}$ is isomorphic to one of the groups in $\{D_{14}, D_{10}, D_8, D_6, {\mathbb{Z}}_2\times {\mathbb{Z}}_2\}$. If $\frac{G}{Z(G)}$ is isomorphic to $D_{14}, D_{10}, D_8$ or $D_6$ then,  by Theorem \ref{main4},  it follows that ${\mathcal{A}}_G$ is  L-integral. If $\frac{G}{Z(G)}$ is isomorphic to ${\mathbb{Z}}_2\times {\mathbb{Z}}_2$ then,  by Theorem \ref{main2}, it follows that ${\mathcal{A}}_G$ is   L-integral. Hence,  the result follows.
\end{proof}
\begin{proposition}
Let $G$ be a finite group and $p$ the smallest prime divisor of $|G|$. If $\Pr(G) = \frac{p^2 + p - 1}{p^3}$ then ${\mathcal{A}}_G$ is   L-integral.
\end{proposition}
\begin{proof}
If $\Pr(G) = \frac{p^2 + p - 1}{p^3}$ then by \cite[Theorem 3]{dM74} we have $\frac{G}{Z(G)}$ is isomorphic to ${\mathbb{Z}}_p\times {\mathbb{Z}}_p$. Now, by Theorem \ref{main2}, it follows that ${\mathcal{A}}_G$ is   L-integral.
\end{proof}
\begin{proposition}
If $G$ is a non-solvable group with $\Pr(G) = \frac{1}{12}$ then ${\mathcal{A}}_G$ is L-integral.
\end{proposition}
\begin{proof}
By \cite[Proposition 3.3.7]{Caste10}, we have that $G$ is isomorphic to $A_5 \times B$ for some abelian group $B$. Since $A_5$ is an AC-group, by Corollary \ref{AC-cor}, it follows that ${\mathcal{A}}_G$ is  L-integral.
\end{proof}

A graph is called planar if it can be embedded in the plane so that no two edges intersect geometrically except at a vertex to which both are adjacent. We conclude this paper with the following result.
\begin{proposition}
Let $G$ be a finite group then ${\mathcal{A}}_G$ is   L-integral if ${\mathcal{A}}_G$ is planar.
\end{proposition}

\begin{proof}
It was shown in Proposition 2.3 of \cite{Ab06} that  ${\mathcal{A}}_G$ is planar if and only if $G$ is isomorphic to $D_6, D_8$ or $Q_8$. Therefore, by Corollary \ref{main005} and Corollary \ref{q4m}, the result follows.
\end{proof}


\end{document}